\documentclass[12pt]{article}

\usepackage{color}
\usepackage[T1]{fontenc}
\usepackage{amsmath}
\usepackage{amssymb}
\usepackage{amsthm,url}
\usepackage{dsfont}
\usepackage{textcomp}
\usepackage{wasysym}
\usepackage{textcomp}
\usepackage{graphicx}

\newtheorem{theorem}{Theorem}
\newtheorem{lemma}[theorem]{Lemma}

\newtheorem{corollary}[theorem]{Corollary}

\newtheorem*{problem*}{Problem}

\theoremstyle{definition}
\newtheorem{remark}[theorem]{Remark}

\renewcommand{\thetheorem}{\arabic{theorem}}

\newcommand{\rk}{\operatorname{rk}}

\begin{document}

\title{On the rank of $\mathbb{Z}_2$-matrices with free entries on the diagonal}

\author{E. Kogan
  \footnote{
    We would like to thank
    R. Fulek,
    B. Mohar,
    V. Podolskii,
    D. Pozdeev,
    A. Skopenkov,
    I. Soloveychik,
    M. Tancer,
    and D. Trushin
    for helpful discussions and suggestions.
  }}
\date{}

\maketitle

\begin{abstract}

    For an $n \times n$ matrix $M$ with entries in $\mathbb{Z}_2$
    denote by $R(M)$ the minimal rank of all the matrices
    obtained by changing some numbers on the main diagonal of $M$.
    We prove that for each non-negative integer $k$
    there is a polynomial in $n$ algorithm deciding whether $R(M) \leq k$
    (whose complexity may depend on $k$).
    We also give a polynomial in $n$ algorithm
    computing a number $m$ such that $m/2 \leq R(M) \leq m$.
    These results have applications to graph drawings on non-orientable surfaces.

\end{abstract}


For an $n \times n$ matrix $M$ with entries in $\mathbb{Z}_2$
denote by $R(M)$ the minimal rank of all the matrices
obtained by changing some numbers on the main diagonal of $M$.
We present two algorithms estimating $R(M)$. These results
(Theorems \ref{main} and \ref{main-approx})
have applications to graph drawings on non-orientable surfaces,
see Appendix A.
For a brief overview of the history of the problem, see Remark \ref{history}.




Denote by $M_n(\mathbb{Z}_2)$ the set of all $n \times n$ matrices with entries in $\mathbb{Z}_2$.

\begin{theorem}\label{main}
    Let $k$ be a fixed non-negative integer.
    There is an algorithm with the complexity of $O(n^{k + 4})$ deciding
    for an arbitrary matrix $M \in M_n(\mathbb{Z}_2)$
    whether $R(M) \leq k$.
\end{theorem}

The proof of Theorem \ref{main} uses Lemma \ref{non-degenerate-algo} and well known Lemma \ref{rk-est}.

\begin{theorem}\label{main-approx}
    There is an algorithm with the complexity of $O(n^4)$
    calculating
    for an arbitrary matrix $M \in M_n(\mathbb{Z}_2)$
    a number $k$ such that
    \begin{equation*}
        k / 2 \leq R(M) \leq k.
    \end{equation*}
\end{theorem}

The proof of Theorem \ref{main-approx} uses Lemma \ref{min-rk-est} which in turn uses well known Lemma \ref{rk-est}.

\begin{remark}
    Let $M$ be an $n \times n$ matrix with entries in $\mathbb{Z}_2$.
    A matrix $M'$ obtained from $M$ by changing some numbers on the main diagonal of $M$
    is called
    \emph{$k$-good} if $\rk M' \leq k$
    and
    \emph{$k$-realising} if $\rk M' = k$.
    The algorithm presented in the proof of Theorem \ref{main}
    can be easily modified to
    calculate the numbers on the main diagonal of a $k$-good matrix
    if such numbers exist.
    The algorithm presented in the proof of Theorem \ref{main-approx}
    can be easily modified to
    calculate the numbers on the main diagonal of a $k$-realising matrix.
\end{remark}

\begin{remark}\label{history}
    %
    %
    There is a related
    {Low-Rank Matrix Completion Problem} (LRMC) which
    has been extensively studied, see survey
    \cite{real-matrix-completion-survey}.

    Consider the following problem which we call $F$-LRMC.

    Let $F$ be a field.
    Let $M$ be an $n \times m$ matrix with entries in $F$.
    Let $\Omega$ be a set of cells of $M$.
    What is the minimal rank of all the $n \times m$ matrices
    obtained from $M$ by changing some entries having indices not in $\Omega$?

    LRMC is a special case of $F$-LRMC with $F = \mathbb{R}$.

    In our paper, a special case of $\mathbb{Z}_2$-LRMC is considered with
    $m = n$ and $\Omega = \overline{\Omega_{diag}} := \{(i, j) \mid i, j \in \{1, \ldots, n\}, i \neq j\}$.

    In \cite{param-matrix-completion},
    $\mathbb{Z}_p$-LRMC is considered.
    In the case of the main diagonal
    having entries to change
    (i.\,e. in the case of $\Omega = \overline{\Omega_{diag}}$)
    the results obtained in \cite{param-matrix-completion}
    give algorithms with larger complexity
    than the algorithm given by Theorem~\ref{main}.

    $\mathbb{Z}_2$-LRMC is NP-hard \cite{np-hardness}.

    In \cite{real-matrix-completion-survey}, LRMC is considered.
    The methods used in \cite{real-matrix-completion-survey}
    give results based on minimizing different approximations of the rank of a matrix,
    for example,
    the sum of the singular values of a matrix \cite{SVD}.
    Thus, none of the results of \cite{real-matrix-completion-survey}
    can be applied to the problem of finding $R(M)$.

    In \cite{random-matrix-completion},
    a problem similar to $\mathbb{Z}_2$-LRMC is considered
    with the difference being
    the assumption that $\Omega$ is formed by choosing cells
    at random independently from each other with the same probability.
    The provided algorithm is probabilistic.

    In \cite{top-matrix-completion}, the following variation of $\mathbb{Z}_2$-LRMC is considered.
    Take a symmetric matrix $M$ with entries in $\mathbb{Z}_2$
    whose rows and columns are indexed by the edges of a connected graph.
    Let $\Omega$ be equal to the set of pairs of independent edges of the graph.
    What is the minimal rank of all \emph{symmetric} matrices
    obtained from $M$ by changing some entries having indices not in $\Omega$?
\end{remark}








\begin{lemma}\label{non-degenerate-algo}
    There is an algorithm with the complexity of $O(n^4)$ which
    for an arbitrary matrix $M \in M_n(\mathbb{Z}_2)$
    finds some numbers from $\mathbb{Z}_2$ to replace the entries on the main diagonal of $M$ with
    such that the resulting matrix $M'$ is non-degenerate.
\end{lemma}

\begin{proof}
    The proof consists of two parts: a description and an estimation of the complexity of the algorithm
    and a proof of that the algorithm gives the required numbers.

    \emph{Part 1. Description and estimation of the complexity of the algorithm.}

    Denote by $a_i$ the desired numbers to put on the main diagonal of $M'$.
    Let us calculate these numbers one by one from top-left to bottom-right. Suppose we have calculated $a_1, \ldots, a_{i - 1}$.
    Then calculate the corner minor $\Delta_i$ with the size $i \times i$
    assuming that a zero is on the $i$-th place on the main diagonal of the matrix $M'$.
    Put $a_i = 1 + \Delta_i$ on the $i$-th place on the diagonal.

    There is an algorithm computing the determinant of a $m \times m$ matrix with the complexity of $O(m^3)$ \cite{Gauss}.
    Then, since the presented algorithm is essentially a computation of the determinants of $n$ submatrices
    with sizes $1 \times 1, 2 \times 2, \ldots, n \times n$, its complexity is $O(1^3 + 2^3 + \ldots + n^3) = O(n^{4})$.

    \emph{Part 2. Proof that the algorithm gives the required numbers.}

    Let us prove by induction by $i$ that after step $i$ of the first part of the algorithm
    the corner minor $\Delta_i'$ of $M'$ with the side $i$ is equal to $1$.

    \emph{Base. $i = 1$.} $\Delta_i = 0 \Rightarrow$ we put $1$ in the first diagonal element. Hence $\Delta_i' = 1$.

    \emph{Step. $i \rightarrow i + 1$.} By the induction hypothesis $\Delta_i' = 1$.
    By the decomposition formula of the determinant $\Delta_{i + 1}'$ by the last row of the corresponding submatrix of $M'$
    \begin{equation*}
        \Delta_{i + 1}' = \Delta_{i + 1} + (1 + \Delta_{i + 1})\Delta_i' = \Delta_{i + 1} + (1 + \Delta_{i + 1}) = 1
    \end{equation*}

    Thus, $\det M' = \Delta_n' = 1$.
\end{proof}

A matrix is called \textbf{diagonal} if all its entries outside of the main diagonal are equal to $0$.

\begin{lemma}\label{rk-est}
    Let $M, D$ be matrices of the same size with entries in $\mathbb{Z}_2$
    such that $D$ is diagonal.
    Then $\rk(M + D) \geq \rk M - \rk D$.
\end{lemma}

\begin{proof}
    Since $D$ has $\rk D$ nonzero rows, upon addition to $M$ it changes $\rk D$ rows.
    Since there are $\rk M$ linearly independent rows in the matrix $M$
    and at least $\rk M - \rk D$ of them remain the same after the addition of $D$, it follows that
    there are $\rk M - \rk D$ linearly independent rows in the matrix $M + D$.
    Hence $\rk (M + D) \geq \rk M - \rk D$.
\end{proof}

\begin{proof}[Proof of Theorem \ref{main}]
    The proof consists of two parts: a description and an estimation of the complexity of the algorithm
    and a proof that the algorithm gives the right answer.

    \emph{Part 1. Description and estimation of the complexity of the algorithm.}

    Since $k$ is fixed, it is sufficient to estimate the complexity of the algorithm for $n \geq 2k$.

    Apply the algorithm given by Lemma \ref{non-degenerate-algo}.
    Denote by $M'$ the non-degenerate matrix given by the applied algorithm.

    There is an algorithm searching through all diagonal $n \times n$ matrices
    with $l$ zeroes and $n - l$ identities on the main diagonal
    with the complexity of $O\left(n\binom{n}{l}\right)$.
    Hence there is an algorithm searching through all diagonal $n \times n$ matrices
    with $\leq k$ zeroes on the main diagonal
    with the complexity of
    \begin{multline*}
        O\left(n\binom n0 + n\binom n1 + \ldots + n\binom nk\right) \\
        = O\left(n(k + 1)\binom n{\min(n/2, k)}\right)
        = O\left(n(k + 1)\binom nk\right)\\
        = O\left(n\cdot n^{k}\right)
        = O\left(n^{k + 1}\right).
    \end{multline*}

    Apply the algorithm searching through all diagonal $n \times n$ matrices with $\leq k$ zeroes on the main diagonal.
    For every such matrix $D$ calculate $\rk(M' + D)$.
    If for at least one such matrix $D$ the rank of $M' + D$ is less than or equal to $k$
    then return that it is possible to achieve a rank $\leq k$.
    Otherwise return that it is not possible to achieve a rank $\leq k$.

    For each matrix $M' + D$ where $D$ is a diagonal matrix with $\leq k$ zeroes on the main diagonal
    the algorithm described above calculates $\rk (M' + D)$.
    Since the rank of a matrix can be calculated by an algorithm with the complexity of $O(n^3)$,
    the total complexity of the algorithm described in the previous paragraph
    is $O(n^{k + 1} n^3) = O(n^{k + 4})$. Thus, the complexity of the whole algorithm is $O(n^4) + O(n^{k + 4}) = O(n^{k + 4})$.

    \emph{Part 2. Proof that the algorithm gives the right answer.}

    Note that any matrix $M''$ formed as a result of putting numbers from $\mathbb{Z}_2$ in the elements on the main diagonal of $M$
    can be uniquely represented as the sum $M' + D$ where $D$ is a diagonal matrix.
    By Lemma \ref{rk-est} every diagonal matrix $D$ with $l < n - k$ identities on the main diagonal satisfies
    \begin{equation*}
        \rk (M' + D) \geq \rk M' - \rk D = n - l > n - (n - k) = k.
    \end{equation*}
    Thus, it is sufficient to search through matrices $D$
    with at least $n - k$ identities on the main diagonal which is exactly what the algorithm does.
\end{proof}

\begin{lemma}\label{min-rk-est}
    Let $M$, $D$ be matrices of the same size with entries in $\mathbb{Z}_2$
    such that $M$ is non-degenerate and $D$ is diagonal.
    Then $\rk (M + D) \geq \rk (M + I) / 2$ where $I$ is the identity matrix.
\end{lemma}

\begin{proof}
    Denote by $n$ the amount of columns of $M$ and $D$.
    By Lemma \ref{rk-est} we have
    \begin{multline*}
        2\rk(M + D) = \rk(M + D) + \rk((M + I) + (I + D)) \\
        \geq (\rk M - \rk D) + (\rk (M + I) - \rk (I + D)) \\
        = n - \rk D + \rk (M + I) - (n - \rk D) = \rk (M + I).
    \end{multline*}
\end{proof}

\begin{proof}[Proof of Theorem \ref{main-approx}]
    Let us describe the algorithm.

    Apply the algorithm given
    by Lemma \ref{non-degenerate-algo}
    to the matrix $M$.
    Denote by $M'$ the matrix obtained by this algorithm.
    Denote $k := \rk (M' + I)$. The number $k$ can be computed in time $O(n^3)$.

    By Lemma \ref{min-rk-est} and the fact that $M' + I$ can be obtained from $M$
    by changing some numbers on the main diagonal,
    \begin{equation*}
        k/2 \leq R(M) \leq k
    \end{equation*}
    as required.

    The total complexity of the algorithm is $O(n^4) + O(n^3) = O(n^4)$.
\end{proof}

\begin{remark}

    Let $M$ be an $n \times n$ matrix with entries in $\mathbb{Z}_2$.
    Then $R(M) \leq n - 1$.

    \begin{proof}
        If we change the numbers on the main diagonal of $M$
        so that the sum of the entries in each row is even,
        the resulting matrix $M'$ will obviously be degenerate.
        Hence $\rk M' \leq n - 1$ and $R(M) \leq n - 1$.


    \end{proof}

\end{remark}

\appendix

\section{Appendix}
\setcounter{theorem}{0}
\renewcommand{\thetheorem}{A\arabic{theorem}}


This appendix
describes the applications of the main results to graph drawings on non-orientable surfaces
(Corollaries \ref{main-top}, \ref{main-top-approx}).
The problem of graph drawings on non-orientable surfaces has been extensively studied
(for example, see \cite{Mohar}).
See the definitions of an hieroglyph, weak realizability, the disk with $k$ Möbius strips below.
Our results are different from \cite{Mohar-article} because
there, realizability is studied which comes down to calculating $\rk M$
(see Theorem \ref{Mohar}),
and we study weak realizability which comes down to calculating $R(M)$.
The paper \cite{topology-base} gives an algorithmic criterion
for the weak realizability of an hieroglyph on the Möbius strip (the disk with 1 Möbius strip).

Let us introduce the notion of weak realizability.

The \textbf{disk with $k$ Möbius strips}
is the figure shown on the left of fig. \ref{fig:disk-example}.\footnote{
  Since any nonorientable 2-surface of nonorientable genus $k$ is homeomorphic to the disk with $k$ Möbius strips,
  in this definition
  the term ``disk with $k$ Möbius strips'' can be replaced with ``non-orientable 2-surface of nonorientable genus $k$''.
}
More precisely, the disk with $k$ Möbius strips
is the union of a disk
and $k$ pairwise disjoint ribbons having their ends glued to $2k$
pairwise disjoint arcs on the boundary circle of the disk
(the ribbons do not have to lie in the plane of the disk)
so that
\begin{enumerate}
    \renewcommand{\labelenumi}{(\theenumi)}
    \renewcommand{\theenumi}{\alph{enumi}}
        \item the orientations of the ends of each ribbon
    given by an orientation of the boundary circle of the disk
    have ``the same direction along the ribbon'', and
        \item the ribbons are ``separated'',
    i.\,e. there are $k$ pairwise disjoint arcs $A_i$ of the boundary circle of the disk
    such that the ends of the $i$-th ribbon are glued to two disjoint arcs
    contained in $A_i$ ($i = 1, 2, \ldots, k$).
\end{enumerate}

\begin{figure}[h]
    \centering
    \includegraphics{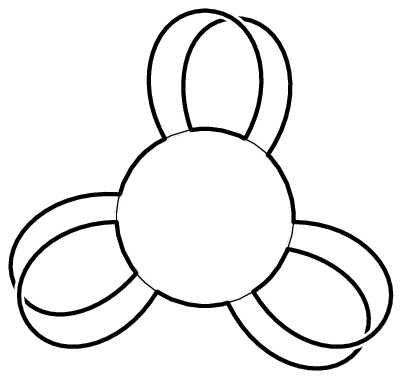}
    \ \ \ \includegraphics{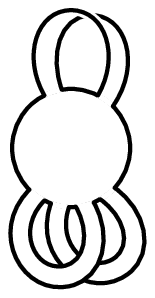}
    \caption{On the left: the disk with 3 Möbius strips
      which is a disk with ribbons corresponding to the hieroglyph $aabbcc$.
      On the right: a disk with ribbons corresponding to the hieroglyph $aabcbc$ which is not the disk
      with 3 Möbius strips (because the bottom ribbons violate each condition of
      the definition of the disk with $k$ Möbius strips)
      but is homeomorphic to the disk with $3$ Möbius strips.}
    \label{fig:disk-example}
\end{figure}

An \textbf{hieroglyph} on $n$ letters is an unoriented cyclic letter sequence of length $2n$
such that each letter from the sequence appears in the sequence twice.

Take a hieroglyph $H$ on $n$ letters.
Take a convex polygon with $2n$ sides.
Put the letters in the hieroglyph
on the sides of the convex polygon in the nonoriented cyclic order.
For each letter glue the ends of a ribbon to the pair of sides corresponding to the letter
so that the glued ribbons are pairwise disjoint.
Call the resulting surface a \textbf{disk with ribbons} corresponding to the hieroglyph $H$
(see fig. \ref{fig:disk-example}).
Since each ribbon can be either twisted or not twisted,
several surfaces may correspond to a single hieroglyph.

A hieroglyph $H$ is \textbf{weakly realizable} on the disk with $k$ Möbius strips
if some disk with ribbons corresponding to $H$ can be cut out of the disk with $k$ Möbius strips.

For a hieroglyph $H$ denote by $R(H)$ the minimal number $k$ such that the hieroglyph $H$
is weakly realizable on the disk with $k$ Möbius strips.
Such a number exists by the classification theorem for compact surfaces with boundary.

\begin{corollary}\label{main-top}
    For each non-negative integer $k$
    there is an algorithm with the complexity of $O(n^{k + 4})$
    which for an arbitrary hieroglyph $H$ on $n$ letters
    decides whether $R(H) \leq k$.
\end{corollary}

\begin{corollary}\label{main-top-approx}
    There is an algorithm with the complexity of $O(n^4)$
    which for an arbitrary hieroglyph $H$ on $n$ letters
    computes a number $k$ such that
    \begin{equation*}
        k/2 \leq R(H) \leq k.
    \end{equation*}
\end{corollary}

Two letters $a, b$ in a hieroglyph $H$ \textbf{overlap in $H$} if they interlace in the cyclic sequence of the hieroglyph
(i.e. they appear in the sequence in the order $abab$ but not $aabb$).
Take an $n \times n$ matrix with entries in $\mathbb{Z}_2$ with zeroes on the main diagonal.
Put $1$ in the cell $(i, j)$ for $i \neq j$ if the letters $i, j$ overlap in $H$ and $0$ otherwise.
Call the resulting matrix
the \textbf{overlap matrix} of the hieroglyph $H$.

\begin{theorem}[Mohar]\label{Mohar}
    Let $M$ be the overlap matrix of a hieroglyph $H$.
    Then $R(H) = R(M)$.
\end{theorem}

Theorem \ref{Mohar} is a corollary of \cite[Theorem 3.1]{Mohar-article}
(see also \cite[\S 2.8, statement 2.8.8(c)]{Skopenkov}).

The following Lemma \ref{top2mat-algo} is well known.

\begin{lemma}\label{top2mat-algo}
    There is an algorithm with the complexity of $O(n^2)$
    which for an arbitrary hieroglyph on $n$ letters constructs its overlap matrix.
\end{lemma}

\begin{proof}
    The following is a rough description of the algorithm.

    Take an $n \times n$ matrix $M$ with zeroes on the main diagonal.
    For each letter $i$ do the following.
    The two occurences of $i$ split the hieroglyph into two sequences of letters $A_i$ and $B_i$.
    Go through all the letters in any one sequence $A_i$ or $B_i$.
    For each occurence of such a letter $j$ add $1$ to the cell $(i, j)$ of $M$.
    Return the resulting matrix.

    The complexity of the presented algorithm is obviously $O(n^2)$.

    The algorithm gives the required matrix since the letters $i, j$ overlap
    if and only if $j$ appears in the sequence $A_i$ or $B_i$ an odd number of times (once).


\end{proof}

\begin{proof}[Proof of Corollary \ref{main-top}]
    The following is a description of the algorithm.

    Denote by $M$ the overlap matrix of the hieroglyph $H$.
    It can be calculated by the algorithm given by Lemma \ref{top2mat-algo}.
    Apply the algorithm given by Theorem \ref{main} to $M$.
    Return the result of the last applied algorithm.

    The total complexity of the presented algorithm is $O(n^2) + O(n^{k + 4}) = O(n^{k + 4})$.

    The presented algorithm gives the right answer by Theorem \ref{Mohar}.
\end{proof}

The \emph{proof of Corollary \ref{main-top-approx}}
can be obtained from the proof of Corollary \ref{main-top}
by replacing Theorem \ref{main} and $O(n^{k + 4})$
by Theorem \ref{main-approx} and $O(n^4)$, respectively.


\begin{remark}
    For a multigraph $G$, a \textbf{half-edge} of $G$ is an orientation of one of its edges.
    For a vertex $v$ of a multigraph
    an (oriented or unoriented) \textbf{local rotation at $v$}
    is an (oriented or unoriented) cyclic ordering
    of the half-edges incident to $v$.
    A multigraph $G$ with a family of unoriented local rotations at its vertices
    is called \textbf{weakly realizable} on the disk with $k$ Möbius strips
    if $G$ can be drawn on the disk with $k$ Möbius strips so that
    for each vertex $v$ of $G$ the local orientation at $v$ matches
    the cyclic ordering at which the half-edges incident to $v$
    intersect the boundary of a small circle around $v$.
    A natural generalization of finding $R(H)$ for a hieroglyph $H$
    is finding the least integer $k$
    such that a given multigraph
    with a family of unoriented local rotations at its vertices
    is weakly realizable on the disk with $k$ Möbius strips.
    However, this problem cannot be reduced to the problem of finding $R(H)$
    by contracting spanning trees of the components of the multigraph
    because upon contracting an edge $(v, u)$,
    it is unclear how to combine the unoriented local orientations at $v$ and $u$.
\end{remark}

\end{document}